\documentclass[
final
]{dmtcs-episciences}

\usepackage[utf8]{inputenc}
\usepackage{subfigure}
\usepackage{array}
\usepackage{tabularx}
\usepackage{makecell}
\usepackage{amssymb,amsmath}
\usepackage{amsfonts}
\usepackage{amsthm}
\usepackage[raggedright]{titlesec}

\usepackage[round]{natbib}

\newcommand{\litem}[1]{\item[#1\hfill]}

\newenvironment{mylist}[1]{
    \setbox1=\hbox{#1}
    \begin{list}{}{
            \setlength{\labelwidth}{\wd1}
            \setlength{\leftmargin}{\wd1}
            \addtolength{\leftmargin}{0em}
            \addtolength{\leftmargin}{\labelsep}
            \setlength{\rightmargin}{1em}}}{\end{list}}

\newcommand\Item[1][]{%
  \ifx\relax#1\relax  \item \else \item[#1] \fi
  \abovedisplayskip=0pt\abovedisplayshortskip=0pt~\vspace*{-\baselineskip}}

\newcommand{\dbg}[3]{cDB^{+}({#1},{#2},{#3})}
\newcommand{\bg}[3]{cDB({#1},{#2},{#3})}

\newtheorem{theorem}{Theorem}
\newtheorem{lemma}{Lemma}
\newtheorem{corollary}{Corollary}

\newtheorem{fact}{Fact}
\newtheorem{definition}{Definition}

\author{Tiziana Calamoneri\affiliationmark{1}
  \and Angelo Monti\affiliationmark{1}
  \and Blerina Sinaimeri\affiliationmark{2}}
  
\title[On the domination number of $t$-constrained de Bruijn graphs]{On the domination number of $t$-constrained de Bruijn graphs\thanks{This work is partially supported by the following research projects: {\em Sapienza} University of Rome, projects: no. RM120172A3F313FE "Measuring the similarity of biological and medical structures through graph isomorphism", no. RM11916B462574AD "A deep study of phylogenetic tree reconciliations" and no. RM1181642702045E "Comparative Analysis of Phylogenies". }}
\affiliation{
  Computer Science Department, Sapienza University of Rome, Italy\\
 Luiss University, Rome, Italy}
\received{2021-12-21}
\revised{2022-06-28}
\accepted{2022-07-05}
\begin{document}
\publicationdetails{24}{2022}{2}{2}{8879}
\maketitle
\begin{abstract}
Motivated by the work on the domination number of directed de Bruijn graphs and some of its generalizations, in this paper we introduce a natural generalization of de Bruijn graphs (directed and undirected), namely \emph{$t$-constrained de Bruijn} graphs, where $t$ is a positive integer, and then study the domination number of  these graphs.

Within the definition of $t$-constrained de Bruijn graphs, de Bruijn and Kautz graphs correspond to 1-constrained and 2-constrained de Bruijn graphs, respectively. 
This generalization inherits many structural properties of de Bruijn graphs and may have similar applications in interconnection networks or bioinformatics.  

We establish upper and lower bounds for the domination number on $t$-constrained de Bruijn graphs both in the directed and in the undirected case. 
These bounds are often very close and in some  cases we are able to find the exact value.

\end{abstract}


\section{Introduction}\label{sec:intro}
In graph theory, the study of domination and dominating sets plays a prominent role.  This topic has been extensively studied for  more than 30 years \cite{Hedetniemi1998,Hedetniemi1990} due to its  applications in several areas, \textit{e.g.} wireless networks \cite{Dai2004},  protein-protein interaction networks \cite{Milenkovic2011},  social networks \cite{Bonato2015}. For a comprehensive treatment of domination and its variations, we refer to  \cite{Hedetniemi1998,Haynes2020,Haynes2021}.

In an undirected graph, a vertex \emph{dominates} itself and all its neighbors.  The concept of domination can be naturally transferred to directed graphs, where a vertex dominates itself and all of its outgoing neighbors.  A \emph{dominating set} of a (directed or undirected) graph is a subset $S$  of vertices such that every vertex in the graph is dominated by at least one vertex in $S$. 
The domination number of a graph $G$ is the cardinality of a smallest dominating set of $G$, and is denoted by $\gamma(G)$. 
Finding a minimum dominating set for general graphs is widely know to be NP-hard \cite{garey1979} and hence it is a challenge to determine classes of graphs for which $\gamma(G)$ can be exactly computed. 

Indeed, finding a closed formula for the domination number is a well-studied problem and has been solved for several classes of graphs such as
directed de Bruijn graphs \cite{Blazsik2002DominatingSI}, directed Kautz graphs \cite{KIKUCHI2003}, generalized Petersen graphs \cite{Petersen_graphs}, Cartesian product of two directed paths \cite{Mollard2014} and graphs defined by two levels of the $n$-cube \cite{BADAKHSHIAN2019}. Furthermore, close bounds are provided for some generalizations of the previous classes \cite{BALOGH2021,Dong2015}.

In this paper we focus on the well-known de Bruijn graphs which have various applications in different areas as for example bioinformatics \cite{Orenstain2017,Pevzner2001}, interconnection networks \cite{Esfahanian1985} and peer-to-peer systems \cite{Koorde}.  De Bruijn graphs are defined as follows:

\begin{definition}
Given an alphabet $\Sigma$ and a positive integer $n$, the {\em directed de Bruijn graph} of dimension $n$ on $\Sigma$ is defined by taking  as vertices all the sequences in $\Sigma^n$ and as edges the ordered pairs of the form $\bigl((a_1, \ldots, a_n), (a_2, \ldots, a_n, a_{n+1})\bigr)$, where $a_1, \ldots, a_{n+1}$ are in $\Sigma$. 
\end{definition}

\begin{figure}[h]
\centering
\includegraphics[scale=0.5]{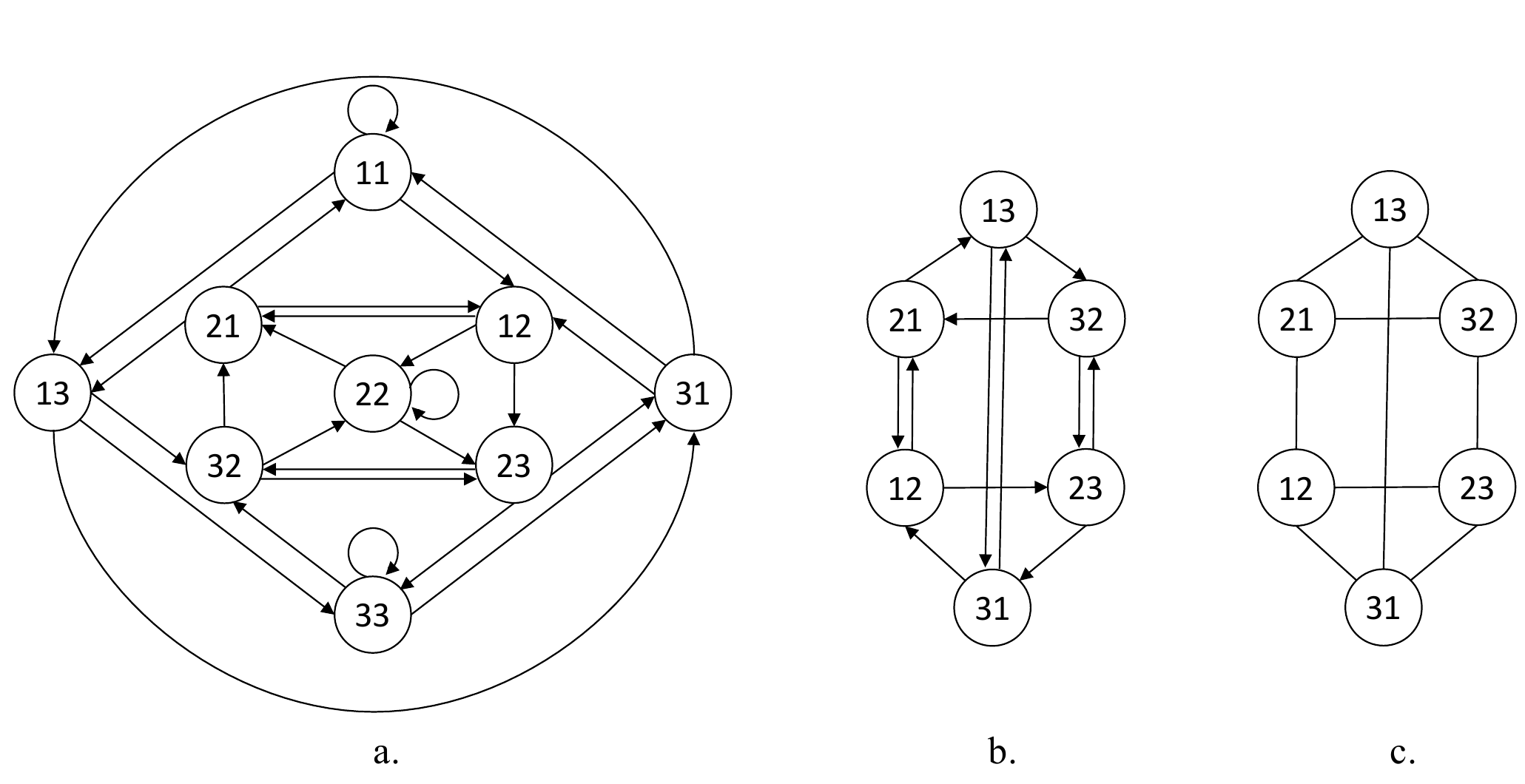}
\caption{For $\Sigma=\{1,2,3\}$  we show: (a) the directed  de Bruijn graph $cDB^+(d,1,2)$; (b) the directed Kautz graph $cDB^+(3,2,2)$; (c) the undirected Kautz graph $cDB(3,2,2)$.
}\label{fig.DB_and_K}
\end{figure}

In Fig.~\ref{fig.DB_and_K}.a we show a directed de Bruijn graph of dimension 2 on $\Sigma=\{1,2,3\}$.   Particular induced subgraphs of de Bruijn graphs have also been studied due to their applications related to both DNA assembly and high-performance or fault-tolerant computing \cite{Li2004}.  These subgraphs can be defined by choosing particular subsets of vertices. For instance, the subgraph induced by the sequences  of $\Sigma^n$ that do not contain equal neighboring characters  corresponds to the well-known directed Kautz graph \cite{Kautz1968} which has many properties that are desirable in computer networks \cite{Li2004,Bermond1993}.  In Fig.~\ref{fig.DB_and_K}.b we show a directed Kautz graph of dimension 2 on $\Sigma=\{1,2,3\}$. The undirected version of de Bruijn and Kautz graphs can be easily obtained by simply ignoring the direction of the edges and removing loops and multiple edges (see Fig.~\ref{fig.DB_and_K}.c as the undirected version of the Kautz graph in Fig.~\ref{fig.DB_and_K}.b). 

In this paper we propose a new natural generalization of the Kautz graphs, obtained by extending the constraint on the sequences labeling the vertices, from neighboring positions to an arbitrary distance $t$.

\begin{definition}\label{def:constrained}
Given a sequence ${\bf x}=x_1, \ldots, x_n \in \Sigma^n$  we say that ${\bf x}$ is {\em $t$-constrained}, for some integer $t$, if for all $1 \leq i < j \leq n$, whenever $x_i=x_j$ it holds $|i-j| \geq t$.
\end{definition}

Notice that every sequence in $\Sigma^n$ is trivially $1$-constrained while the sequences labeling Kautz graphs are $2$-constrained. Moreover, the non trivial cases of Definition~\ref{def:constrained} are when $1 \leq t \leq \min\{d,n \}$.  Indeed, if $t>\min\{d,n \}=n$ then any sequence in $\Sigma^n$ is trivially $t$-constrained, whereas if $t>\min\{d,n \}=d$ none of the sequences of $\Sigma^n$ is $t$-constrained.

We denote by $V(d,t,n)$ the set of \emph{all} $t$-constrained sequences from the set $\Sigma^n$. We now introduce $t$-constrained de Bruijn graphs. 

\begin{definition}
Given an alphabet $\Sigma$, with $|\Sigma|=d$,  and two positive integers $n$ and $t$, with $1 \leq t \leq \min\{d, n \}$, we define the \emph{directed $t$-constrained de Bruijn graph of dimension $n$ on $\Sigma$} as the subgraph of the directed de Bruijn graph of dimension $n$ on $\Sigma$ induced by the set  $V(d,t,n)$.
\end{definition}

We denote a $t$-constrained directed de Bruijn graph  with $\dbg{d}{t}{n}$ and its  undirected version with $\bg{d}{t}{n}$. In Fig.~\ref{fig.3deBruijin}, $\bg{4}{3}{4}$ is depicted.
Clearly, directed  de Bruijn graphs and directed Kautz graphs coincide with $\dbg{d}{1}{n}$ and $\dbg{d}{2}{n}$, respectively. 

In addition to their theoretical interest, $t$-constrained de Bruijn graphs may also find applications in interconnection networks where  it is important to design network topologies offering a high-level of symmetries. Indeed, it is easier to balance the traffic load, and hence to minimize the congestion, on network topologies with a high-level of symmetry.  

In bioinformatics area, $t$-constrained de Bruijn graphs could be more suitable than de Bruijn graphs in  modelling problems. For example, in genome rearrangement,  permutations of integers (\textit{i.e.} $n$-constrained sequences) are used to represent genomes \cite{Alekseyev2007,Lin2014}. However, it is known that genes often undergo duplication, thus the genome may contain different copies of the same gene. In this context   $t$-constrained sequences for $t < n$ may be used to model genomes where duplication of genes is allowed only at a certain distance in the genome. 

\begin{figure}[!h]
\centering
\includegraphics[scale=0.6]{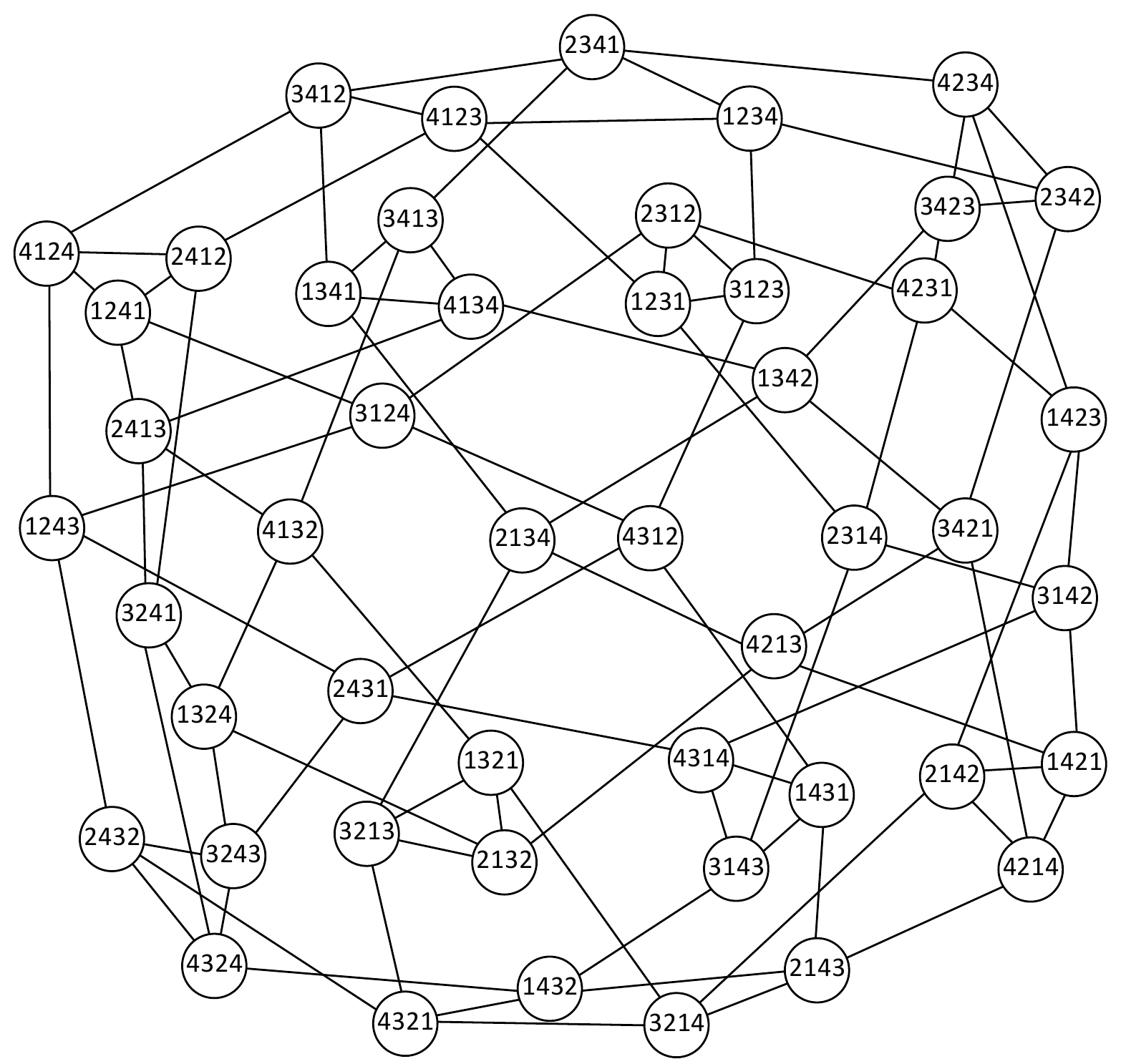}
\caption{For $\Sigma=\{1,2,3,4\}$ the constrained de Bruijn graph $cDB(4,3,4)$.
}\label{fig.3deBruijin}
\end{figure}

In this paper we provide a systematic study of the domination number of $t$-constrained de Bruijn graphs in both directed  and undirected case.  

Although the domination numbers for the directed case of de Bruijn  and Kautz graphs have been exactly determined  (see \cite{KIKUCHI2003,Blazsik2002DominatingSI,ARAKI2007}), the exact values in the undirected cases are still missing.  Here we provide close upper and lower bounds on the value of  $\gamma(\bg{d}{1}{n})$ and $\gamma(\bg{d}{2}{n})$. 
Furthermore, in the particular case when the sequences labeling the vertices are permutations (\textit{i.e.}  $t=n$), we determine the exact value of the domination number in both the directed and undirected case.  

We also consider the case where the sequences are partial $n$-permutations on the set of symbols when $|\Sigma|=d=n+c$, for some integer $c$. In this case we provide close upper and lower bounds for $\gamma(\bg{n+c}{n}{n})$ and $\gamma(\dbg{n+c}{n}{n})$.

Finally, we provide upper and lower bounds for the domination number of $\dbg{d}{t}{n}$.  Concerning the value of $\gamma(\bg{d}{t}{n})$ it remains an open problem to find an upper bound that is asymptotically better than the one trivially derived by the directed case. The results of this paper are summarized in Table~\ref{tab:results_directed} and Table~\ref{tab:results_undirected} for the directed and undirected cases, respectively. 

\begin{table}[h!]
\centering
\small
\setlength\tabcolsep{2pt} 
\setlength\extrarowheight{6pt}
\begin{tabular}{|l|r|c|}
\hline
\multicolumn{1}{|l|}{Graph $G$} & \multicolumn{1}{c|}{$\gamma(G)$}
 & \multicolumn{1}{c|}{ Ref.} \\
\hline
$\dbg{d}{1}{n}$ & 
$\gamma(G)=\left \lceil{\frac{d^n}{d+1}}\right \rceil$  &  \cite{Blazsik2002DominatingSI}  \\ \hline  

$\dbg{d}{2}{n}$  &
$\gamma(G)=(d-1)^{n-1}$ & \cite{KIKUCHI2003} \\ \hline 

$\dbg{d}{3}{n}$ & \makecell{$\gamma(G) = d(d-2)^{n-2}$  if $d$ even \\ $d(d-2)^{n-2}
\leq \gamma(G)\leq \Big(1+\Theta\big(\frac{1}{d^2}\big)\Big)d(d-2)^{n-2}$ if  $d$ odd } & {Thm.\ref{th.directed_d3n}} \\ \hline


$\dbg{d}{t}{n}$ & {\small
$\frac{d!}{(d-t)!}\frac{(d-t+1)^{n-t}}{(d-t+2)}
\leq \gamma(G)\leq  \Big(1+ \Theta\big( \frac{t}{d(d-t+1)}\big)\Big)\frac{d!}{(d-t)!}\frac{(d-t+1)^{n-t}}{(d-t+2)}$} &
Thm.\ref{th.directed_dtn} \\ \hline


%
$\dbg{n}{n}{n}$ &
$\gamma(G)=\left\lceil\frac{n}{2}\right\rceil(n-1)!$ &
Thm.\ref{theo:directed_NNN} \\ \hline

$\dbg{n+c}{n}{n}$ &
$\frac{1}{c+2} \frac{(n+c)!}{c!} \leq \gamma(G)  \leq \Big( 1+\Theta\big(\frac{1}{c}\big)\Big) \frac{1}{c+2}\frac{(n+c)!}{c!}$ & Thm.\ref{theo:directed_NCNN} \\ \hline

\end{tabular}
\caption{Summary of the results for directed $t$-constrained de Bruijn graphs.}\label{tab:results_directed}
\end{table}


\begin{table}[h!]
\centering
\small
\setlength\tabcolsep{2pt} 
\setlength\extrarowheight{6pt}

\begin{tabular}{|l|r|r|}
\hline

\multicolumn{1}{|l}{Graph $G$}&\multicolumn{1}{|c|}{$\gamma(G)$ }
 & \multicolumn{1}{c|}{ Ref.} \\
\hline

$\bg{d}{1}{2}$ &
$\gamma(G)=d-1$ & 
Thm.\ref{theo:bruijn_n_2} \\

$\bg{d}{1}{3}$  & 
$\gamma(G)=d \bigl\lceil\frac{d}{2}\bigr\rceil$ &
Thm.\ref{theo:bruijn_n_3}\\

$\bg{d}{1}{n}$, $n \geq 4$  &
$\frac{d^{n}}{2d+1}  \leq \gamma(G)\leq \Big( 2 - \Theta\big(\frac{1}{d}\big) \Big) \frac{d^{n}}{2d+1}$
&
Thm.\ref{theo:deBruij} \\

\hline 

$\bg{d}{2}{2}$  &
$\gamma(G) = d-1$ &
Thm.\ref{theo:kautz_n_2} \\

$\bg{d}{2}{3}$  &
$\frac{d(d-1)}{2}\leq \gamma(G) \leq \bigl\lfloor  \frac{d^2}{2} \bigr\rfloor$ &
Thm.\ref{theo:kautz_n_3} \\

$\bg{d}{2}{n}, n \geq 4$  &
$
\frac{d(d-1)^{n-1}}{2d-1} \leq \gamma(G)\leq \Big(2- \Theta\big(\frac{1}{d}\big) \Big)\frac{d(d-1)^{n-1}}{2d-1}
$
&
Thm.\ref{theo:Kautz} \\

\hline

$\bg{d}{3}{n}$ &
$
\frac{d(d-1)(d-2)^{n-2}}{2d-3}\leq \gamma(G)\leq \Big(2- \Theta\big(\frac{1}{d}\big) \Big) \frac{d(d-1)(d-2)^{n-2}}{2d-3}
$ &
Thm.\ref{theo:undirected_d3n} \\ \hline


$\bg{n}{n}{n}$ &
$\gamma(G)=\left\lceil\frac{n}{3}\right\rceil(n-1)!$ &
Thm.\ref{theo:undirected_NNN} \\ \hline

$\bg{n+c}{n}{n}$ &
$ \frac{1}{2c+3}\frac{(n+c)!}{c!} \leq \gamma(G) \leq \Bigl(1+ \Theta\bigl(\frac{1}{c}+ \frac{1}{n}\bigr)\Bigr) \frac{1}{2c+3}\frac{(n+c)!}{c!}$ & Thm.\ref{theo:undirected_NCNN} \\ \hline

\end{tabular}
\caption{Summary of the results for undirected $t$-constrained de Bruijn graphs.}
\label{tab:results_undirected}
\end{table}

\section{Preliminaries}\label{sec:preliminaries}

Let $G$ be an undirected graph with $V(G)$ and $E(G)$ its vertex and edge set, respectively.  Let $v\in V(G)$;  $N(v)$ denotes the \emph{neighborhood} of $v$, \textit{i.e.} the set of vertices that are adjacent to $v$. 
Similarly, for a directed graph we denote by $N^{+}(v)$ and $N^{-}(v)$, the out- and in-neighborhood, respectively.
We define $\Delta(G)$ as the $\max_v \{|N(v)|\}$ if $G$ is undirected and the $\max_v \{|N^+(v)|\}$ if $G$ is directed.

Given an undirected  graph $G$, a set of vertices $S \subseteq V(G)$ is a \emph{dominating set} if, for every vertex $v$, either $v \in S$ or  $N(v)\cap S \neq \emptyset$.  Hence, we will say that $v$ \emph{dominates} $N(v)\cup\{v\}$. The dominating set is defined similarly for directed graphs by considering $N^{+}(v)$ instead of $N(v)$. 

The minimum over the cardinality of all dominating sets of (un)-directed $G$ is denoted by $\gamma(G)$.

For any graph $G$, a straightforward lower bound on $\gamma(G)$, following by the definition of a dominating set, is 
\begin{equation}
\label{eq:grado}
   \gamma(G)\geq \left \lceil \frac{|V(G)|}{\Delta(G)+1}\right \rceil .
\end{equation}

The sequences representing the vertices of our graphs are on an alphabet $\Sigma$ with $d$ symbols; 
without loss of generality we assume that $\Sigma=\{1,2, \ldots, d\}=[d]$.
The concatenation of two strings $\textbf{x}, \textbf{y}$ is denoted $\textbf{x}\cdot \textbf{y}$.

As already mentioned, de Bruijn and Kautz graphs coincide with $\dbg{d}{1}{n}$ and  $\dbg{d}{2}{n}$, respectively. 
More in general, we have the following hierarchy: 
$$\dbg{d}{n}{n}, \, \dbg{d}{n-1}{n}, \, \ldots,\, \dbg{d}{1}{n}$$ where each graph in the sequence is a proper subgraph of its successor. The same hierarchy holds also for the undirected case.

From here on we assume that $n\geq 2$ and $d \geq 2$. Indeed, when $d=1$ the corresponding graph contains a single node,
while if $n=1$ both $\dbg{d}{1}{1}$ and $\bg{d}{1}{1}$ 
are complete graphs on $d$ vertices and hence their domination number is 1.

\begin{fact}\label{fact:lower_degree}
Given a directed $t$-constrained de Bruijn graph $cDB^+(d,t,n)$ and its undirected version $cDB(d,t,n)$ we have:
\begin{itemize}
\item 
the cardinality of the vertex set for both graphs is:
\begin{equation}
\label{eq:vertices}
|V(d,t,n)|=\frac{d!}{(d-t)!}(d-t+1)^{n-t}
\end{equation}

\item
$\Delta(\dbg{d}{t}{n})=d-t+1$ and $\Delta(\bg{d}{t}{n})\leq 2(d-t+1)$

\item
from the previous two items and from Equation~\ref{eq:grado} we have:

\begin{equation}
\label{eq:lower_bound_grado_diretto}
\gamma(cDB^+(d,t,n)) \geq \left \lceil \frac{d!}{(d-t+2)!}(d-t+1)^{n-t+1} \right \rceil 
\end{equation}

and 
\begin{equation}
\label{eq:lower_bound_grado_nondiretto}
\gamma(cDB(d,t,n)) \geq \left \lceil \frac{d!(d-t+1)^{n-t}}{(d-t)!(2d-2t+3)}\right \rceil.
\end{equation}
\end{itemize}
\end{fact}

Note that, for $t\geq 2$,  the value in (\ref{eq:lower_bound_grado_diretto}) is always an integer and thus, the ceiling can be removed.

\section{Domination numbers of $\dbg{d}{1}{n}$ and $\bg{d}{1}{n}$}
\label{sec.DeBruijn}
In this section we consider $1$-constrained de Bruijn graphs, which correspond to the well-known de Bruijn graphs.  The domination number in the directed case has been studied  in \cite{Blazsik2002DominatingSI} proving that the exact value matches the lower bound in (\ref{eq:lower_bound_grado_diretto}). For the sake of completeness we recall the result in the next theorem.

\begin{theorem}
\label{theo:bruijn_directed}
\cite{Blazsik2002DominatingSI}
For any two integers $d\geq 2$ and $n\geq 2$ it holds
$$
\gamma(\dbg{d}{1}{n})=\left \lceil{\frac{d^n}{d+1}}\right \rceil.
$$
\end{theorem}

On the other hand the undirected case is still open. Indeed, to the best of our knowledge only the perfect domination (\textit{i.e.} the case when each vertex is dominated by exactly one vertex) has been studied for undirected de Bruijn graphs \cite{Blazsik2002DominatingSI,Livingston1990}.

Here we provide lower and upper bounds on $\gamma(\bg{d}{1}{n})$ for any $n \geq 2$.  We start by considering the special cases $n=2$ and $n=3$ for which we provide formulas.

\begin{theorem}\label{theo:bruijn_n_2}
For any integer $d \geq 2$, $\gamma(\bg{d}{1}{2})=d-1$.
\end{theorem}
\begin{proof} 
{\em Upper bound. }
Observe that the vertices of $\bg{d}{1}{2}$ are all the sequences  $(x_1,x_2) \in [d]^2$. 
Define the subset $S=\{(1,x): x \in [d], x\neq 1\}$ of cardinality $d-1$. We show that $S$ is a dominating set. Let $(x_1,x_2) \in [d]^2$ be any of the vertices of $\bg{d}{1}{2}$. There are three possible cases to consider: (i) $x_2=1$, then any sequence of $S$ dominates it, (ii) $x_2\neq 1$ but $x_1=1$, then the sequence $(x_1,x_2)$ belongs to $S$ and thus is trivially dominated, (iii) $x_2\neq 1$ and $x_1\neq 1$,  then $(1,x_1)$ is in $S$ and dominates $(x_1,x_2)$. Thus, $\gamma(\bg{d}{1}{2})\leq d-1$.   

\smallskip

{\em Lower bound. }
The lower bound deduced from (\ref{eq:lower_bound_grado_nondiretto}) gives only $\left \lceil \frac{d^2}{2d+1} \right \rceil$ and thus we need to use a different argument to prove our claim. 

Let $S$ be a dominating set. Let $P_1 ,P_2$ the set of symbols that, in the sequences in $S$, appear  in the first and second coordinates, respectively. 
Let $M_1= [d] -P_1$ and $M_2= [d]- P_2$ be the sets of symbols that, in the sequences in $S$,  do not appear in the first and second coordinates, respectively. 
Observe that $M_1 \cap M_2 =\emptyset $, {\em i.e.} every symbol $a \in [d]$ appears in some sequence in $S$, otherwise it would not be possible to dominate sequence $(a,a)$ where $a \in M_1 \cap M_2$.
For $i=1,2$, let $|P_i|=p_i$ and $|M_i|=m_i$ and obviously  $p_1+m_1=p_2+m_2=d$. Notice that for any $i=1,2$, it holds $|S| \geq |P_i|$ and thus if $m_i=0$ we are done. 

Assume then $m_i \geq 1$ and observe that, for any two symbols $a \in M_1$ and $b\in M_2$, sequence $(b,a)$ must belong in $S$. 
Indeed the only sequences that can dominate $(b,a)$ are: \\ 
(i) sequences of the form $(x,b)$ which cannot be in $S$ as $b \in M_2$, \\
(ii) sequences of the form $(a,y)$ which cannot be in $S$ as $a \in M_1$, \\
(iii) sequence $(b,a)$ itself. 

Thus, there are $m_1 m_2$ pairs $(b,a)$ in $S$, with $a \in M_1$ and $b\in M_2$. Clearly these sequences have in the first position a symbol in $M_2$ and thus contribute with exactly $m_2$ symbols in $P_1$. 
Hence, the rest of the $p_1-m_2$ symbols in $P_1$ must appear as a first coordinate in $|S|-m_1 m_2$ sequences. 
So $p_1-m_2\leq |S| - m_1m_2$. 
Since $p_1=d-m_1$, we have:
$$|S| \geq d+ m_1 m_2 -m_1 - m_2= d +(m_1-1)(m_2-1) - 1 \geq d-1$$
\noindent
where the last inequality follows by observing that $m_1$ and  $m_2$ are both positive integers. 
\end{proof}

\begin{theorem}\label{theo:bruijn_n_3}
For any integer $d \geq 2$, $\gamma(\bg{d}{1}{3})=d \bigl\lceil\frac{d}{2}\bigr\rceil$.
\end{theorem}
\begin{proof} 
Also in this case we handle separately upper and lower bounds.

\smallskip

{\em Upper bound. }
Assume first $d$ even and let  $S=\bigl\{(2i,x, 2i-1): 1\leq i\leq \frac{d}{2}, x \in [d]\bigr\}$.
Clearly,  $|S|=\frac{d^2}{2}$. To show that $S$ is a dominating set observe that, given any sequence $(a,b,c) \in [d]^3$, either  $b$ is even (and then it is dominated by sequence $(b,c,b-1) \in S$), or $b$ is odd (and then it is dominated by sequence $(b+1,a,b) \in S$). 

When $d$ is odd the argument is the same but in this case we consider set  $S=\bigl\{(2i,x, 2i-1): 1\leq i\leq \frac{d-1}{2}, x \in [d]\bigr\} \cup \bigl\{(d,x,d): x\in [d] \bigr\}$. 
Clearly, $|S|=d \bigl\lceil\frac{d}{2}\bigr\rceil $.
We now prove  that $S$ is a dominating set. First, from the previous paragraph it is clear that $S$ dominates all sequences $(a,b,c)$ that do not contain $d$. 
If $(a,b,c)$ contains at least  two symbols equal to $d$, then sequences  $(d,d,x)$ and $(x,d,d)$ are dominated by $(d,d,d) \in S$ since all the sequences $(d,x,d)$ are in $S$. 
If, finally, $(a,b,c)$ contains only one $d$, we distinguish three cases according to the position where $d$ appears: \\
(i) $(a,b,d)$ is dominated by $(b,d,b-1)$ when $b$ is even or by $(b+1,a,b)$ when $b$ is odd; \\
(ii) $(a,d,c)$ is dominated by $(d,c,d)$; \\
(iii) $(d,b,c)$ is dominated by $(b+1,d,b)$ when $b$ is odd or by $(b,c,b-1)$ when $b$ is even. 

\smallskip

{\em Lower bound. } The lower bound deduced from  (\ref{eq:lower_bound_grado_nondiretto})  gives only $\bigl\lceil\frac{d^3}{2d+1}\bigr\rceil\leq d\bigl\lceil\frac{d}{2}\bigr\rceil -\bigl\lfloor\frac{d}{4}\bigr\rfloor+1$ and thus we need to use a different argument to prove our claim. 
Let $S$ be a dominating set and for any symbol $a$ we denote by $S_a$ the sequences of $S$ of the form $(x,a,y)$ (\textit{i.e.} those that have the symbol $a$ in the second position). Clearly, if for every symbol $a$ it holds that $|S_a| \geq  \bigl\lceil\frac{d}{2}\bigr\rceil $ we are done as $|S|=\sum_{a \in [d]}|S_a| \geq d \bigl\lceil\frac{d}{2}\bigr\rceil$. Then assume there exists a symbol $a$ for which $|S_a|\leq \bigl\lceil\frac{d}{2}\bigr\rceil -1$. We denote by $P^{1}_a$ the set of pairs $(x,a)$ for which there exists a sequence $(x,a,y)$ in $S_a$. Similarly we denote by $P^{2}_a$ the set of pairs $(a,y)$ for which there exists a sequence $(x,a,y)$ in $S_a$. 
By definition of  $P^{1}_a$ and  $P^{2}_a$ we have that  $|P^{1}_a| \leq |S_a| \leq \bigl\lceil\frac{d}{2}\bigr\rceil -1$ and $|P^{2}_a| \leq |S_a| \leq \bigl\lceil\frac{d}{2}\bigr\rceil -1$. Now, we show that it must necessarily hold that $(a,a) \in  P^{1}_a \cup  P^{2}_a$. Indeed, the triple $(a,a,a)$ can be dominated only by the following triples: $(a,a,a)$, $(z,a,a)$, $(a,a,z)$ with $z \in [d]$. Thus, without loss of generality, we can assume $(a,a) \in  P^{1}_a$ (the proof is identical if $(a,a) \in  P^{2}_a$).

Let $M$ be the set of the pairs $(x,a)$ that do not appear in the first two positions  of the sequences in $S$.  We have $|M| = d- | P^{1}_a| \geq d - \bigl\lceil\frac{d}{2}\bigr\rceil +1 \geq \bigl\lceil\frac{d}{2}\bigr\rceil$.  We denote the pairs in $M$ as $(x_1,a), (x_2,a), \ldots, (x_m,a)$ with $m\geq \bigl\lceil\frac{d}{2}\bigr\rceil$. Notice that $(a,a) \not\in M$, thus for any $1\leq i\leq m$, $x_i \neq a$.

Fix an $1\leq i\leq m$ and consider sequences $(z,x_i,a)$ for every $z\in [d]$.  
As these sequences are dominated by $S$, for each $z$, at least one between $(z,x_i,a)$ and $(b,z,x_i)$ belongs to $S$ (note that $(x_i,a,c)$ does not belong to $S$ as pair $(x_i,a) \in M$). Moreover as $x_i\neq a$, these sequences are all distinct.
We deduce that, for any pair  $(x_i,a)$, there are $d$ different triples (one for each $z$) that belong to $S$. 
This holds for any $x_i$ with $1\leq i\leq m$. 
Moreover for any $i\neq j$ and symbols $z, z', b, b'$ not necessary distinct among them, it holds $\{(z,x_i,a),(b,z,x_i)\} \cap \{(z',x_j,a),(b',z',x_j)\}=\emptyset$. 
Thus, $|S| \geq dm \geq d \bigl\lceil\frac{d}{2}\bigr\rceil$.
\end{proof}

We now prove  a general result for any $n\geq 4$.  

\begin{theorem}\label{theo:deBruij}
For any two integers $d\geq 2$ and $n\geq 4$ it holds:
$$
\frac{d^{n}}{2d+1}  \leq \gamma(\bg{d}{1}{n})\leq (d-1)  d^{n-2}= \bigg( 2 - \Theta\Big(\frac{1}{d}\Big) \bigg) \frac{d^{n}}{2d+1}.$$
\end{theorem}

\begin{proof}
{\em Lower bound. } The lower bound follows from  (\ref{eq:lower_bound_grado_nondiretto}). 

{\em Upper bound. } To prove the upper bound, consider the following set:
$$
S = \bigl\{(x_1, x_2, \ldots, x_{n}) \in [d]^n: x_{n-1}\neq x_1,  x_{n}=x_2   \bigr\}.
$$

Notice that $S$ is well defined since $n \geq 4$. Clearly, the cardinality of $S$ is $(d-1)d^{n-2}$.  Indeed, once $x_1, x_2,\ldots, x_{n-2}$ are fixed, there are $d-1$ possibilities for $x_{n-1}$ and one possibility for $x_n$.  
We show  that $S$ is a dominating set for $\bg{d}{1}{n}$.  Suppose, by contradiction, that there is a sequence $\textbf{s}=(x_1, x_2, \ldots, x_{n-1}, x_{n})$ that is not dominated by any vertex in $S$. Then $\textbf{s} \not \in S$ and both the following conditions must occur:

\begin{enumerate}
    \item For any $y \in [d]$, sequence $\textbf{s}'=(y, x_1, x_2, \ldots, x_{n-2}, x_{n-1}) \not\in S$, otherwise $\textbf{s}'$ dominates $\textbf{s}$. In particular, the sequence $(y, x_1, x_2, \ldots, x_{n-2}, x_{n-1}) \not\in S$ for $y\neq x_{n-2}$. By the definition of $S$, we must necessarily have $x_1\neq x_{n-1}$.
  
    \item For any symbol $y \in [d]$, sequence $\textbf{s}''=(x_2, x_3, \ldots, x_{n-1}, x_n, y) \not\in S$, otherwise $\textbf{s}''$ would dominate $\textbf{s}$. In particular if we set $y=x_3$, the sequence $(x_2, x_3, \ldots, x_n, x_3) \not\in S$. By definition of $S$, we must necessarily have $x_2= x_{n}$.
\end{enumerate}
From items 1. and 2. we have that $x_1\neq x_{n-1}$ and $x_2= x_{n}$ and thus  $\textbf{s}$ must have the form $(x_1, x_2, \ldots, z, x_2)$ with $z\neq x_1 $. Thus, $\textbf{s}\in S$ by the definition of $S$, contradicting our initial hypothesis.
\end{proof}

\paragraph*{Comparison with the directed case.}
We conclude this section by observing that a dominating set for a directed graph is {\em a fortiori} also a dominating set for its underlying undirected graph. 
Thus, $\gamma(\bg{d}{1}{n})\leq \gamma(\dbg{d}{1}{n})$. However,  in this section we improved this bound.  
Namely, for  $n=2$ and $n=3$, the upper bound deduced from  Theorem~\ref{theo:bruijn_directed} is not tight as shown by our formulas for $\gamma(\bg{d}{1}{2})$ and $\gamma(\bg{d}{1}{3})$.
For  $n\geq 4$ the upper bound provided by studying the directed case is $\left\lceil\frac{d^{n}}{d+1}\right\rceil$ and we improved it in  Theorem~\ref{theo:deBruij} to $\Big \lceil \big(1-\frac{1}{d^2}\big) \frac{d^n}{d+1}\Big \rceil$.

\section{Domination numbers of $\dbg{d}{2}{n}$ and $\bg{d}{2}{n}$}\label{sec:kautz}

Here we consider $2$-constrained de Bruijn graphs which correspond to Kautz graphs. As for the de Bruijn graphs the  domination number of directed Kautz graphs has been exactly determined in  \cite{KIKUCHI2003} as stated by the following theorem

\begin{theorem}\label{theo:kunz_directed}\cite{KIKUCHI2003}
For any two integers $d\geq 2$ and $n\geq 1$ it holds:
$$\gamma(\dbg{d}{2}{n})=(d-1)^{n-1}.$$
\end{theorem}
\noindent 
Notice that the lower bound in (\ref{eq:lower_bound_grado_diretto}) for $\gamma(\dbg{d}{2}{n})$ is tight.

We consider now the undirected case and similarly to the previous section,  we start by proving exact results for $n=2$.
 
\begin{theorem}\label{theo:kautz_n_2}
For any integer $d \geq 2$,
$\gamma(\bg{d}{2}{2})=d-1.$
\end{theorem}
\begin{proof} 
{\em Upper bound. } The upper bound follows from Theorem \ref{theo:bruijn_n_2} observing that a dominating set for $\bg{d}{1}{2}$ is also a dominating set for $\bg{d}{2}{2}$. 

\smallskip

\noindent
{\em Lower bound. } Concerning the lower bound, let $S$ be a dominating set for $\bg{d}{2}{2}$. Let $C$ be the set of symbols of $\Sigma$ that do not appear in any sequence in $S$. Then necessarily $|C|< 2$; indeed, if by contradiction there were $a, b \in C$, with $a\neq b$, $S$ could not dominate the vertex corresponding to the sequence $(a,b)$, contradicting the hypothesis of $S$ being a dominating set. 

If $|C|=1$, then let $C=\{a\}$.
Then, for any $x\neq a$, $S$ must dominate both $(a,x)$ and $(x,a)$, and so there must exists $y_1, y_2 \in \Sigma$ such that $(x,y_1)$ and $(y_2,x)$ are both in $S$.
So every symbol $x\neq a$ must appear both in the first and in the second position in $S$, and thus $|S|\geq d-1$. 

Finally, if $|C|=0$, for any $a\in \Sigma$ there is a sequence between $(a,x)$ and $(x,a)$ in $S$ for some $x\neq a$. 
This sequence dominates also $(a,a)$ and thus $S$ is also a dominating set for $\bg{d}{1}{2}$.  Thus, from Theorem~\ref{theo:bruijn_n_2} we have $|S| \geq \gamma(\bg{d}{1}{2}) =d-1$ and this concludes the proof.
\end{proof}

\begin{theorem}\label{theo:kautz_n_3}
For any integer $d \geq 2$,
$$
\frac{d(d-1)}{2}\leq \gamma(\bg{d}{2}{3}) \leq \biggl\lfloor  \frac{d^2}{2} \biggr\rfloor .
$$

\end{theorem}
\begin{proof}
\emph{Upper bound:}  We prove first the upper bound by modifying the construction in the proof of Theorem~\ref{theo:bruijn_n_3}.   Suppose first $d$ even. We define the set $S=S_1\cup S_2 \cup  S_3$ of $2$-constrained sequences where: 
\begin{align*}
S_1 &=\bigl\{(b,x, b-1): b\in [d], b \textrm{ is even}, x \in [d] \bigr\}\\
S_2 &= \{(b,b-1,2): b\in [d], b \textrm{ is even} \}, \\
S_3 &= \{(1,b+1,b): b\in [d], b \textrm{ is odd}\} 
\end{align*}

\noindent
It holds that $|S_1|=(d-2)\frac{d}{2}$, $|S_2|=|S_3|=\frac{d}{2}$ and the three sets are mutually disjoint, hence $|S|=d \frac{d}{2}=\bigl\lfloor  \frac{d^2}{2} \bigr\rfloor$. 
It remains to show that $S$ is a dominating set. To this purpose we consider the following two properties:

\begin{mylist}{$({\cal P}_1)$:}
\litem{(${\cal P}_1$)} For any pair $(x,y)$ with $x$ even, there exists in $S$ a sequence of the form $(x,y,a)$ for some $a\in[d]$, $a\neq y$;
\litem{$({\cal P}_2)$} For any pair $(x,y)$ with $y$ odd, there exists in $S$ a sequence of the form $(a,x,y)$ for some $a\in[d]$, $a\neq x$.
\end{mylist}

Clearly, $S$ satisfies ${\cal P}_1$ and ${\cal P}_2$ due to the sequences in $S_1\cup S_2$ and in $S_1\cup S_3$, respectively. 
Now consider an arbitrary $2$-constrained sequence $\textbf{x}=(x_1,x_2,x_3)$ in $\bg{d}{2}{3}$ and hence $x_1\neq x_2, x_2\neq x_3$. If $x_2$ is even then from  ${\cal P}_1$ there is a sequence $(x_2,x_3,a)$ in $S$ which dominates $\textbf{x}$. Otherwise $x_2$ is odd and then from  ${\cal P}_2$ there is a sequence $(a,x_1,x_2)$ in $S$ which dominates $\textbf{x}$. It follows that $S$ is a dominating set.

\medskip

When $d$ is odd, the reasoning is slightly different although it follows the same idea. We define the set  $S=S_1\cup S_2 \cup  S_3\cup S_4$ of $2$-constrained sequences where:
\begin{align*}
S_1 &=\bigl\{(b,x, b-1):  b\in [d], b \textrm{ is even}, x \in [d]\bigr\}\\
S_2 &= \{(b,b-1,d): b\in [d], b \textrm{ is even}\} \\
S_3 &= \{(d,b+1,b): b\in [d-1], b \textrm{ is odd}\}\\
S_4 &= \{(d,b+1,b): b\in [d-1], b \textrm{ is even} \}
\end{align*}

It holds that $|S_1|=(d-2)\frac{d-1}{2}$, $|S_2|=|S_3|=|S_4|=\frac{d-1}{2}$ and the four sets are mutually disjoint, hence $|S|= \frac{(d-1)(d+1)}{2}=\bigl\lfloor  \frac{d^2}{2} \bigr\rfloor$. 
To show that $S$ is a dominating set, note again that $S$ satisfies both properties ${\cal P}_1$ and ${\cal P}_2$ due to the sequences in $S_1\cup S_2$, in $S_1\cup S_3$ if the pair $(x,y)$ does not contain the symbol $d$. To  handle the pairs containing $d$ we define the following further property:

\begin{mylist}{$({\cal P}_1)$:}
\litem{(${\cal P}_3$)} For any $x\in [d-1]$ there exists in $S$ a sequence of the form $(d,x,a)$ for some $a\in[d]$, $a\neq x$;
\end{mylist}

Notice that $S$ satisfies also ${\cal P}_3$ due to pairs in $S_3\cup S_4$. Now consider an arbitrary $2$-constrained sequence $\textbf{x}=(x_1,x_2,x_3)$ in $\bg{d}{2}{3}$ and hence $x_1\neq x_2, x_2\neq x_3$. If $x_2$ is even then, from  ${\cal P}_1$, there is a sequence $(x_2,x_3,a)$ in $S$ which dominate $\textbf{x}$. 
If $x_2$ is odd and $x_2\neq d$ then from  ${\cal P}_2$ there is a sequence $(a,x_1,x_2)$ in $S$ which dominates $\textbf{x}$. Otherwise $x_2=d$ and then from  ${\cal P}_3$ there is a sequence $(d,x_3,a)=(x_2,x_3,a)$ in $S$ which dominates $\textbf{x}$.

\medskip

\emph{Lower bound:} 
Consider the set $R \subset V(d,2,3)$, defined as $R =\{(a,b,a): a\in [d], b\in[d], a\neq b\}$. Clearly, $|R|=d(d-1)$. Let $S$ be a dominating set and notice that any sequence $(x,y,z) \in S$ dominates exactly two elements of $R$. Indeed, if $z\neq x$, $(x,y,z)$ dominates $(y,x,y)$ and $(y,z,y)$, otherwise if $z=x$, it dominates $(x,y,x)$ and $(y,x,y)$. Thus, as all the elements of $R$ must be dominated we have that $|S|\geq \frac{|R|}{2}= \frac{d(d-1)}{2}.$

This concludes the proof.\end{proof}

We now prove a general result for $cDB(d,2,n)$ with $n\geq 4$.

\begin{theorem}\label{theo:Kautz}
For any two integers $d\geq 2$ and $n\geq 4$ it holds:
\begin{eqnarray}\notag
\centering
\frac{d(d-1)^{n-1}}{2d-1} \leq \gamma(\bg{d}{2}{n}) &\leq &  (d-1)^{n-1} -(d-2)(d-1)^{n-4} \\ \notag &=& \bigg(2- \Theta\Big(\frac{1}{d}\Big) \bigg)\frac{d(d-1)^{n-1}}{2d-1}.
\end{eqnarray}
\end{theorem}

\begin{proof}
\emph{Lower bound:}
It follows from  (\ref{eq:lower_bound_grado_nondiretto}). 
%

\emph{Upper bound:} We consider the set $S=S_1 \cup S_2 \cup S_3$ of $2$-constrained sequences where
\begin{align*}
S_1&=\bigl\{ (1, i, x_3,\ldots, x_n) \in [d]^n: i\neq d \bigr\}\\
S_2&=\bigl\{ (1, d, i, x_4\ldots, x_n) \in [d]^n: i\neq 1 \bigr\}\\
S_3&=\bigl\{ (d, 1, d, 1, x_5\ldots, x_n) \in [d]^n\bigr\}
\end{align*}
\noindent
Since $S_1, S_2 ,S_3$ are pairwise disjoint, then, 
$$|S|=(d-2)(d-1)^{n-2}+(d-2)(d-1)^{n-3}+(d-1)^{n-4}= (d-1)^{n-1} -(d-2)(d-1)^{n-4}.$$

It remains to prove that $S$ is a dominating set.
When $d=2$ we have $S_1=S_2=\emptyset$ and $S_3$ contains only the sequence $(d,1,d,1, \ldots)$; trivially $S$ is a dominating set for $\bg{2}{2}{n}$ (which contains only two vertices) and thus we assume $d\geq 3$. 

Given  any $2$-constrained sequence ${\bold y}=(y_1,y_2,\ldots, y_n)$ representing a vertex of $\bg{d}{2}{n}$, we distinguish three cases according to the symbol $y_1$:
\begin{itemize}
\item $y_1=1$: if $y_2\neq d$ then ${\bold y}$ is  a sequence in  $S_1$; if $y_2=d$,  we consider the value of  the third symbol $y_3$. If $y_3 \neq 1$ then ${\bold y}$ is  in $S_2$.  Otherwise,  $y_3=1$ and  ${\bold y}$ is dominated by a sequence in $S_3$.
\item $1< y_1<d$: ${\bold y}$ is dominated by sequence $(1,y_1, \ldots, y_{n-1})$ in $S_1$. 
\item $y_1=d$: if $y_2\neq 1$ then ${\bold y}$ is dominated by a sequence in $S_2$; if $y_2=1$, again we distinguish two cases according to the third symbol.  
If $1< y_3 < d $ then ${\bold y}$ is dominated by $(1, y_3, y_4,\ldots, y_n, z)$ in $S_1$;
if $y_3=d$ we look at $y_4$: if $y_4=1$ then ${\bold y}=(d, 1, d, 1 ,\ldots, y_n)$ is a sequence in $S_3$, otherwise ${\bold y}=(d, 1, d, y_4,\ldots, y_n)$ is dominated by a sequence $(1, d, y_4, \ldots, y_n, z)$ in $S_2$.
\end{itemize}
\end{proof}

\paragraph*{Comparison with the directed case.}
We conclude this section by observing that Theorem~\ref{theo:kunz_directed} 
gives us $ \gamma(\bg{d}{2}{n}) \leq  \gamma(\dbg{d}{2}{n}) = (d-1)^{n-1}$.
 The results of this section show that, for $n=2$, $\gamma(\bg{d}{2}{n})=\gamma(\dbg{d}{2}{n})=d-1$.  For $n\geq 3$ the upper bound of $(d-1)^{n-1}$  derived from the directed case is improved to $(d-1)^{n-1} -(d-2)(d-1)^{n-4}$.
 

\section{Domination numbers for $\dbg{d}{3}{n}$ and $\bg{d}{3}{n}$ }

Here we consider the $3$-constrained de Bruijn graphs. This class of graphs has not been studied in the literature. Similar to the de Bruijn and Kautz graphs, finding the domination number in the undirected case seems more difficult as shown by the next results.

\begin{theorem}
\label{th.directed_d3n}
For any two integers $d\geq  3$ and $n\ge 4$ it holds:
\begin{itemize}
    \item  For $d$ even:  $\gamma(\dbg{d}{3}{n})=d(d-2)^{n-2}$; 
    \item \textrm{For $d$ odd: } \begin{eqnarray} \notag
d(d-2)^{n-2}\leq \gamma(\dbg{d}{3}{n}) &\leq &  (d-1)^2(d-2)^{n-3} \\ \notag &=& \bigg(1+\Theta\Big(\frac{1}{d^2}\Big)\bigg)d(d-2)^{n-2}.
\end{eqnarray}
\end{itemize}
\end{theorem}
\begin{proof}
\emph{Lower Bound.} From (\ref{eq:lower_bound_grado_diretto}) we have:

$$\gamma(\dbg{d}{3}{n})\geq  \biggl\lceil\frac{d(d-1)(d-2)^{n-2}}{d-1}\biggr\rceil\geq
d(d-2)^{n-2}.$$
\medskip
\noindent
\emph{Upper bound. } 
We distinguish two cases, according to the parity of $d$.

If $d$ is even, we define the following set of 3-constrained sequences: $$S=\{ (i,i+1, x_3, \ldots, x_n): i \mbox{ is odd}\} \cup \{ (i, i-1, x_3, \ldots, x_n): i \mbox{ is even} \}.$$
Given any sequence $\textbf{s}=(x_1, \ldots, x_n)$, it is clearly dominated, indeed:
\begin{itemize}
    \item 
if either $x_1$ is even and $x_2=x_1 -1$ or $x_1$ is odd and $x_2=x_1+1$, then $\textbf{s} \in S$;
    \item
    if $x_1$ is even but $x_2 \neq x_1 -1$, $\textbf{s}$ is dominated by $((x_1-1), x_1, \ldots) \in S$ and if $x_1$ is odd but $x_2 \neq x_1 +11$, $\textbf{s}$ is dominated by $((x_1+1), x_1, \ldots) \in S$, regardless of the value of $x_3$.
    \end{itemize}
It is easy to see that $|S|$ coincides with the lower bound.

If, on the contrary, $d$ is odd, the previously defined $S$ is not well defined when $i=d$.
We hence define $S \subseteq V(d,3,n)$ as $S=\{ (i, i+1, x_3, \ldots, x_n) : i \mbox { is odd and }i \neq d\} \cup \{ (i, i-1, x_3, \ldots, x_n) : i \mbox{ is even }\} \\ 
\cup \{ (x, d, x_2, x_3, \ldots, x_n) : x=\min([d] \setminus \{x_2\} \}$.

Let $\textbf{x}=(x_1,x_2,\ldots, x_n)$ be a vertex in $\dbg{d}{3}{n}$. We show that  $\textbf{x}$ is dominated by considering the following three cases:

\begin{itemize}
    \item 
    if $x_1$ is even and $x_2=x_1 -1$ then $\textbf{s} \in S$; if instead $x_2 \neq x_1 -1$ then $\textbf{s}$ is dominated by $((x_1-1), x_1, \ldots) \in S$;
    \item
    if $x_1 \neq d$ is odd and $x_2=x_1+1$, then $\textbf{s} \in S$;
    if, instead, $x_2 \neq x_1 +1$, $\textbf{s}$ is dominated by $((x_1+1), x_1, \ldots) \in S$;
    \item
    if  $x_1=d$ then $\textbf{x}$ is dominated by sequence $(x, d, x_2, x_3, \ldots, x_n)$, that is in $S$.
    \end{itemize}

For what concerns $|S|$, it is $(d-1)(d-2)^{n-2}+(d-1)(d-2)^{n-3}$, that is equal to 
$(d-1)^2(d-2)^{n-3}$ as the stated upper bound. 
\end{proof}

We consider now the undirected case for which we provide lower and upper bounds.

\begin{theorem}\label{theo:undirected_d3n}
For any two integers $d\geq  3$ and $n\ge 4$ it holds:
\begin{eqnarray} \notag
\centering
\frac{d(d-1)(d-2)^{n-2}}{2d-3}\leq \gamma(\bg{d}{3}{n})&\leq& (d-1)(d-2)^{n-2}\\ \notag &=& \biggl(2- \Theta\Big(\frac{1}{d}\Big) \biggr) \frac{d(d-1)(d-2)^{n-2}}{2d-3}.
\end{eqnarray}
\end{theorem}
\begin{proof}
{\em Lower bound.}
The lower bound follows from (\ref{eq:lower_bound_grado_nondiretto}). 

\noindent
{\em Upper bound.} Let $S$ be a set of 3-constrained sequences, and $S = \{(1,x_2,\ldots, x_n) \}$.

To prove that $S$ is a dominating set, consider any vertex \textbf{x}$=(x_1,x_2,\ldots, x_n)$ in $\bg{d}{3}{n}$; we distinguish some cases according to the position $i$ of the first occurrence of 1, the first symbol of $\Sigma$:
\begin{itemize}
    \item 
if $i=1$ then $\textbf{x} \in S$;
\item
if $i=2$ then sequence $(1, x_3, \ldots, x_n, y)$ for any $y \neq x_n, x_{n-1}$ is in $S$ and dominates $\textbf{x}$;
\item
if $3 \leq i \leq n$, then sequence $(1, x_1, \ldots, x_{i-1},1,x_{i+1}, \ldots, x_{n-1})$ is in $S$ and dominates $\textbf{x}$; 
\item
if  $\textbf{x}$ does not contain any 1, then $(1, x_1, \ldots, x_{n-1})$ is in $S$ and dominates $\textbf{x}$.
\end{itemize}
Moreover $|S|=(d-1)(d-2)^{n-2}$.
\end{proof}

\paragraph*{Comparison with the directed case.} 
Since $ \gamma(\bg{d}{3}{n}) \leq  \gamma(\dbg{d}{3}{n})$, for $d$ even Theorem~\ref{th.directed_d3n} gives an upper bound of $d(d-2)^{n-2}$ and Theorem~\ref{theo:undirected_d3n} improves this result to $(d-1)(d-2)^{n-2}=d(d-2)^{n-2}- (d-2)^{n-2}$.  For $d$ odd, Theorem~\ref{th.directed_d3n} gives $(d-1)^2(d-2)^{n-3}$ and Theorem~\ref{theo:undirected_d3n} improves this to $(d-1)(d-2)^{n-2}=(d-1)^2(d-2)^{n-3}-(d-1)(d-2)^{n-3}$.

\section{Domination numbers of $\dbg{d}{t}{n}$ and $\bg{d}{t}{n}$}
In this section we consider the general case $t\geq 3$ and we assume $n>t$. The case $t=n$, corresponding to vertices labeled by (partial) permutations,  will be studied in the next section.  We start by the study of the domination number in the directed case.

\begin{theorem}
\label{th.directed_dtn}
For any three integers $d, t, n$  such that $2\leq t\leq d$ and $t<n$, it holds:
{\small
\begin{eqnarray}\notag
\frac{d!}{(d-t)!}\frac{(d-t+1)^{n-t}}{(d-t+2)}\leq\gamma(\dbg{d}{t}{n})  &\leq&\frac{(d-1)(d-1)!}{(d-t)!}(d-t+1)^{n-t-1} \\ \notag
&=&\biggl(1+ \Theta\Big( \frac{t}{d(d-t+1)}\Big)\biggr)\frac{d!}{(d-t)!}\frac{(d-t+1)^{n-t}}{(d-t+2)} 
\end{eqnarray}
}
\end{theorem}
\begin{proof}
{\em Lower bound. } It follows from (\ref{eq:lower_bound_grado_diretto}) and it is equal to
$\frac{|V(d,t,n)|}{d-t+2}$.  

{\em Upper bound. } Consider the following $t-1$ sets of $t$-constrained sequences:
\begin{itemize}
\item $S_1=\{\textbf{x}=(x_1,x_2,\ldots, x_n) : x_1=1, \textbf{x } \textrm{ is $t$-constrained} \}$;
\item $S_i=\bigl\{\textbf{x}=(x_1,x_2,\ldots, x_n): x_{i}=1, x_1=\min([d] \setminus \{x_2,\ldots, x_t\}), \textrm{\textbf{x} is $t$-constrained}\bigr\}$, for $3\leq i\leq t$.
\end{itemize}
It turns out that $|S_1|=\frac{(d-1)!}{(d-t)!}(d-t+1)^{n-t}$ and 
$|S_i| =\frac{1}{(d-t+1)}\frac{(d-1)!}{(d-t)!}(d-t+1)^{n-t}$ for each $i\geq 3$.
Calling $S=S_1 \cup \bigl(\cup_{i=3}^{t}S_i\bigr)$, we have:
\begin{eqnarray}\notag
|S|& =& \frac{(d-1)!}{(d-t)!}(d-t+1)^{n-t} +\frac{t-2}{(d-t+1)}\frac{(d-1)!}{(d-t)!}(d-t+1)^{n-t}\\ \notag
&=&\frac{(d-1)(d-1)!}{(d-t)!}(d-t+1)^{n-t-1}\\ \notag
& = & \left(1+ \frac{t-2}{d(d-t+1)}\right)\frac{d!}{(d-t)!}\frac{(d-t+1)^{n-t}}{(d-t+2)} \\ \notag 
&=& \biggl(1+ \Theta\Big( \frac{t}{d(d-t+1)}\Big)\biggr)\frac{d!}{(d-t)!}\frac{(d-t+1)^{n-t}}{(d-t+2)}. \notag
\end{eqnarray}

To conclude the proof, it remains to show that $S$ is a dominating set for $\dbg{d}{t}{n}$.
To this aim, consider a general vertex $\bold{x}=(x_1,x_2,\ldots, x_n)$.
If symbol 1 does not occur in the first $t$ coordinates, then $\bold{x}$ is dominated by $(1,x_1, x_2,\ldots, x_{n-1})$ of $S_1$.
If, on the contrary, there exists $1 \leq i \leq t$ for which $x_i=1$ then $i$ must be unique and we distinguish the following three cases:
\begin{itemize}
\item if $i=1$ then $\bold{x} \in S_1$ and hence it is dominated;
\item  if $2\leq i < t$ then $(x, x_1, x_2,\ldots, x_{i-1},1,x_{i+1},\ldots, x_{n-1}) \in S_{i+1}$ with $x=\min([d]-{x_1,x_2\ldots, x_t})$ dominates $\bold{x}$;
\item if $i=t$ then $(1, x_1, x_2,\ldots, x_{t-1},1,x_{t+1},\ldots, x_{n-1}) \in S_1$ dominates $\bold{x}$.
\end{itemize}
\end{proof}

Notice then when $t=2$ the upper bound converges to the lower bound and we obtain the result of Theorem~\ref{theo:kunz_directed}. 

The domination number of the undirected $t$-constrained de Bruijn graphs  seems more difficult and the only bounds we have follow from (\ref{eq:lower_bound_grado_nondiretto}) and  Theorem~\ref{th.directed_dtn}. For the sake of completeness we state them in the next corollary.

\begin{corollary}

\label{theo:undirected_DTN}
For any three integers $d\geq  3$, $t \geq 3$ and $n\ge 4$ it holds
{\small
\begin{eqnarray}\notag
\frac{d!}{(d-t)!} \frac{(d-t+1)^{n-t}}{2d-2t+3}\leq \gamma(\bg{d}{t}{n}) &\leq&\frac{(d-1)(d-1)!}{(d-t)!}(d-t+1)^{n-t-1} \\ \notag
&=& \biggl(2+\Theta\Bigl(\frac{2t-d}{d(d-t+1)}\Bigr)
 \biggr)  \frac{d!}{(d-t)!} \frac{(d-t+1)^{n-t}}{2d-2t+3}
\end{eqnarray}
}
\end{corollary}

Notice that when $d$ goes to infinity the upper bound converges to $2$ times the lower bound unless $d-t$ is bounded above by a constant. Indeed, in the latter case the gap between the lower and upper bound can be larger with the upper bound converging  to $2+\Theta(1)$ times the lower bound.

%
\section{Domination numbers of $\dbg{d}{n}{n}$ and $\bg{d}{n}{n}$}

Here we consider the case $t=n$.  Notice that the set of sequences corresponding to the vertices of these graphs is the set of partial $n$-permutations on the set of symbols $[d]$.  Clearly, we must have $d\geq n$ and thus in this section we write $d=n+c$, for some integer $c\geq 0$. 
Recall that, from (\ref{eq:vertices})  $|V(n+c,n,n)|=\frac{(n+c)!}{c!}$; moreover, the maximum degrees of $\dbg{n+c}{n}{n}$ and of $\bg{n+c}{n}{n}$ are $c+1$ and $2c+2$, respectively.

Given $\dbg{n+c}{n}{n}$ (or $\bg{n+c}{n}{n}$) we define its \emph{$A$-partition} as the tuple of $n+1$ sets $(A_0, A_1,\ldots, A_n)$ where $A_0$ is the set of all vertices that do not contain symbol $d$ whereas for any \mbox{$1\leq i\leq n$}, $A_i$ is the set of all vertices that contain symbol $d$ in position $i$. The cardinality of $A_0$ and $A_i$ are given by the following equations:
 
\begin{equation} \label{eq:cardinality_A0_A1}
|A_0|=\frac{(n+c-1)!}{(c-1)!} \qquad \textrm{and } \qquad |A_i|=\frac{(n+c-1)!}{c!}.
\end{equation}

\subsection{Domination number of $\dbg{n+c}{n}{n}$}\label{subsec:directed_DNN}

We start by considering the special case $d=n$, (\textit{i.e.} $c=0$).  The vertices of the graph $\dbg{n}{n}{n}$ correspond to the permutations of $[n]$. For this graph we exactly determine its domination number.

\begin{theorem}\label{theo:directed_NNN}
For any integer $n\geq 2$ it holds:
$\gamma(\dbg{n}{n}{n})=\left\lceil\frac{n}{2}\right\rceil(n-1)!
$
\end{theorem}
\begin{proof}
To prove the claim it is sufficient to notice that $\dbg{n}{n}{n}$ is the disjoint union of $(n-1)!$ directed cycles of length $n$. Clearly, from (\ref{eq:lower_bound_grado_diretto}), to dominate each of  these cycles we need at least $\left\lceil \frac{n}{2}\right\rceil$ vertices. This number of vertices is also sufficient by taking alternately the vertices in each of the cycles of the graph. 
\end{proof}

Consider the graph $\dbg{n+c}{n}{n}$ and notice that for any vertex $v\in A_i$, with $2\leq i\leq n$, it holds $N^+(v) \subseteq A_{i-1}$.  For any $2\leq i\leq n$, we define a \emph{block} of $A_i$, as the set of vertices in $A_i$ that have the same out-neighborhood, \textit{i.e.} they differ only in the first position.
As there are only \mbox{$d-(n-1)=c+1$} possibilities to choose the first symbol of each fixed sequence, each block has cardinality $c+1$. Furthermore, for any two blocks $B$ and $B'$ of $A_i$, it must hold $B\cap B'=\emptyset$. Indeed, by definition, a block contains all the sequences coinciding in the last $n-1$ positions, so a block can be uniquely identified by a sequence of length $n-1$. Hence, given a sequence $x_1,x_2\ldots,x_{n}$, it will only belong to the block that contains the sequences $a, x_2\ldots,x_{n}$ with $a \in [d]$. Thus the blocks form a partition of the vertices in $A_i$. 
Finally, from these arguments and Equation~\ref{eq:cardinality_A0_A1} we deduce that each $A_i$ can be partitioned in exactly $\frac{(n+c-1)!}{(c+1)!}$ blocks (notice that for any $n\geq 2$ this fraction is an integer). 

We prove the following lemma.

\begin{lemma} \label{lem:a-partition-directed}
Let $(A_0,A_1,\ldots, A_n)$ be the $A$-partition of $\dbg{n+c}{n}{n}$, for any set  $A_i$, $2\leq i\leq n$, there exists a subset $S_i \subset A_i$  of $\frac{|A_i|}{c+1} = \frac{(n+c-1)!}{(c+1)!}$ vertices that dominate all the vertices in $A_{i-1}$.
\end{lemma}
\begin{proof}
Fix an integer $i$ such that $2\leq i\leq n$ and let $B_1, \ldots, B_{r}$ be the blocks of $A_i$. We define $S_i$ by choosing any vertex from each block $B_j$, with $1\leq j \leq r$. Clearly, $|S_i|=r=\frac{(n+c-1)!}{(c+1)!}$ and it remains to show that $S_i$ dominates $A_{i-1}$. Let $\mathbf{x}=(x_1,\ldots, x_{i-2},d,x_i, \ldots, x_n)$ be a vertex in $A_{i-1}$; $\mathbf{x}$ is dominated by any vertex of the type $(a, x_1,\ldots, x_{i-2},d,x_i, \ldots, x_{n-1})$ where $a$ is a symbol not occurring in the first $n-1$ positions of $\textbf{x}$ and $d$ is in the position $i$. Hence, every vertex in $A_{i-1}$ is dominated by some vertex in $A_i$. Any two vertices in the same block of $A_i$ dominate the same vertices in $A_{i-1}$ and hence $S_i$ dominates $A_{i-1}$.
\end{proof}

The following theorem is the main result of this subsection.

\begin{theorem}\label{theo:directed_NCNN}
For  any integer $c\geq 1$ and  $n \geq 2$ it holds:
\begin{align}
\notag
\frac{1}{c+2}\frac{(n+c)!}{c!} &\leq \gamma(\dbg{n+c}{n}{n}) \\ \notag
& \leq  \frac{(n+c-1)(n+c-1)!}{(c+1)!} = \left( 1+\Theta\Big(\frac{1}{c}\Big)\right) \frac{1}{c+2}\frac{(n+c)!}{c!}. 
\end{align}
\end{theorem}
\begin{proof}
{\em Lower bound. } It follows from (\ref{eq:lower_bound_grado_diretto}).

{\em Upper bound. } To prove the upper bound, let $(A_0,A_1,\ldots, A_n)$ be the $A$-partition of the graph $\dbg{n+c}{n}{n}$. For any $3\leq i\leq n$, we consider set $S_i \subset A_i$ as in Lemma~\ref{lem:a-partition-directed}. 
We define the following set 
$$S=\Big(\bigcup_{3\leq i\leq n}S_i\Big) \cup A_1.$$
We show that $S$ is a dominating set.  First, all the vertices in $A_2,\ldots, A_{n-1}$ are dominated by the vertices in $S_3, \ldots, S_n$ using Lemma~\ref{lem:a-partition-directed}. 
The vertices in  $A_1$ are dominated by themselves in $S$ and the vertices in $A_0 \cup A_n$ are dominated by the vertices in $A_1$. Indeed, every vertex $(x_1,x_2,\ldots, x_n) \in A_0$ is dominated by $(d,x_1,\ldots, x_{n-1})$ that belongs in $A_1$ and each vertex $(x_1,\ldots, x_{n-1},d)$ in $A_n$ is dominated by $(d,x_1,\ldots, x_{n-1}) \in A_1$.  From Lemma~\ref{lem:a-partition-directed} and Equation~\ref{eq:cardinality_A0_A1} the number of vertices in $S$ is given by the following:

\begin{eqnarray*}
(n-2)\frac{(n+c-1)!}{(c+1)!} + \frac{(n+c-1)!}{c!}  & =& \left( \frac{n-2}{(n+c)(c+1)}+ \frac{1}{n+c}\right)\frac{(n+c)!}{c!}\\
&=& \biggr( 1+ \frac{1}{c+1} - \frac{c+2}{(n+c)(c+1)}\biggl) \frac{(n+c)!}{(c+2)c!} \\
&=& \left( 1+\Theta\Big(\frac{1}{c}\Big)\right) \frac{1}{c+2}\frac{(n+c)!}{c!}
\end{eqnarray*}
\end{proof}

\subsection{Domination number of $\bg{n+c}{n}{n}$}\label{subsec:undirected_DN+cN}
Similarly as in the previous section  we start by considering the special case $c=0$.  For this graph we exactly determine its domination number.

\begin{theorem}\label{theo:undirected_NNN}
For any integer $n\geq 2$ it holds:
$\gamma(\bg{n}{n}{n})=\left\lceil\frac{n}{3}\right\rceil(n-1)!$

\end{theorem}
\begin{proof}
To prove the claim it is sufficient to notice that $\bg{n}{n}{n}$ is the disjoint union of $(n-1)!$ undirected cycles of length $n$. From (\ref{eq:grado}) to dominate a single cycle we need at least $\left\lceil \frac{n}{3}\right\rceil$ vertices. A dominating set of this cardinality can be obtained by choosing one vertex out of three in any traversal of the cycle.
\end{proof}

\noindent
We consider now graph  $\bg{n+c}{n}{n}$, $c \geq 1$.  Note that, for any vertex $v\in A_i$, with $2\leq i\leq n-1$, it holds  $N(v) \subseteq A_{i-1}\cup A_{i+1}$. We define now the blocks of $A_i$ in the undirected case. For any $2\leq i \leq n-1$ a set $B\subset A_i$ is called a \emph{$u$-block}
of $A_i$ if any two vertices of it differ only in the first or last position. 
More formally, fixing a subsequence $\textbf{z} \in [d]^{n-2}$, a $u$-block of $A_i$ is a set of vertices of the type $x_1\cdot \textbf{z} \cdot x_2$. Notice that each block $B$ of $A_i$ can be identified by the $(n-2)$-length substring that is shared among all the sequences in $B$, and hence the $u$-blocks form a partition of $A_i$. 
Furthermore, once we fix the sequence $\textbf{z}$ that is shared among the sequences in $B$, $x_1, x_2$ must be chosen among the $d-(n-2)=c+2$ elements that do not appear in $\textbf{z}$ and hence there are exactly $(c+2)(c+1)$ possibilities to choose $x_1, x_2$. 
Thus, each $u$-block has cardinality  $(c+2)(c+1)$ and for any $2\leq i \leq n-1$, from Equation~\ref{eq:cardinality_A0_A1} there are exactly $\frac{|Ai|}{(c+2)(c+1)}=\frac{(n+c-1)!}{(c+2)!}$ blocks in $A_i$.

We now prove a result similar to Lemma~\ref{lem:a-partition-directed} in the undirected case. 

\begin{lemma} \label{lem:a-partition_undirected}

Let $(A_0,A_1,\ldots, A_n)$ be the $A$-partition of $\bg{n+c}{n}{n}$,  for any set  $A_i$, $2\leq i\leq  n-1$ there exists a subset $S_i \subset A_i$  of $\frac{(n+c-1)!}{(c+1)!}$ vertices that dominates all the vertices in $A_{i-1} \cup A_{i+1}$.
\end{lemma}
\begin{proof}
Fix an $2\leq i \leq n-1 $ and let $B_1,\ldots, B_r$ be the $u$-block partition of $A_i$. 
Consider $B_j$ with $1\leq j \leq r$, let $\textbf{z}$ be the shared sequence in $B_j$ and let  $y_1<y_2<\ldots <y_s$ be the $s=c+2$ possible symbols for the first and last position of the sequences in $B_j$. 
We define a set $C_j \subset B_j$ as follows:
$$C_j=\{y_2\cdot \textbf{z} \cdot y_1, \quad  y_3 \cdot \textbf{z} \cdot y_2,\quad \ldots, \quad y_s \cdot \textbf{z} \cdot y_{s-1},\quad y_1\cdot \textbf{z} \cdot y_s \}$$

We define $$S_i=\bigcup_{1\leq j \leq r} C_j. $$ 
Clearly, as  $|C_j|=s=c+2$ and the $u$-blocks form a partition of $A_i$, we have $|S_i|=(c+2) \frac{(n+c-1)!}{(c+2)!} = \frac{(n+c-1)!}{(c+1)!}$. 
It remains to show that $S_i$ dominates all the vertices in $A_{i-1} \cup A_{i+1}$.

 \begin{itemize}
 \item $S_i$ dominates $A_{i-1}$:  Consider an arbitrary vertex $\textbf{x}=(x_1,\ldots, x_n) \in A_{i-1}$ with $x_{i-1}=d$; it can be dominated by any vertex $(a, x_1,\ldots, x_{n-1}) \in A_{i}$ for some symbol $a$ that does not appear in the $(n-1)$-prefix of $\textbf{x}$. 
 Sequence $(a, x_1, \ldots, x_{n-1})$ can be written in the form $a \cdot \textbf{z} \cdot x_{n-1}$ which belongs to the $u$-block $B_j$ whose sequences share subsequence $\textbf{z}$. 
 Clearly, by construction of $C_j$, for any fixed $\textbf{z} \cdot x_{n-1}$ there always exists an $a$ for which $a \cdot \textbf{z} \cdot x_{n-1} \in C_j$.  
 
\item $S_i$ dominates $A_{i+1}$:  The argument is similar to the previous item. Let  $\textbf{x}=(x_1, \ldots, x_n) \in A_{i+1}$ with $x_{i+1}=d$; $\textbf{x}$ can be dominated by any vertex $(x_2, \ldots, x_{n}, a) \in A_{i}$ for some symbol $a$ that does not appear in the $(n-1)$-suffix of $\textbf{x}$. 
Sequence $(x_2, \ldots, x_{n}, a)$ can be written in the form $x_2\cdot  \textbf{z} \cdot a$ which belongs to the block $B_j$ whose sequences share the subsequence $\textbf{z}$. 
Again, by construction of $C_j$, for any fixed $x_2 \cdot \textbf{z}$ there always exists an $a$ for which $ x_2\cdot \textbf{z} \cdot a \in C_j$.
\end{itemize}
\end{proof}

\noindent
Since the case $n=2$ is solved exactly by Theorem~\ref{theo:kautz_n_2}, in the following theorem we assume $n\geq 3$. 

\begin{theorem}\label{theo:undirected_NCNN}
For any integer $c\geq 1$ and $ n \geq 3$ it holds:
{\small
\begin{align*}
\frac{1}{2c+3}\frac{(n+c)!}{c!} &\leq \gamma(\bg{n+c}{n}{n}) \\ \notag
& \leq  \left(1+ \frac{1}{2c+2}+ \frac{4}{n}+\frac{2}{n(c+1)} -  \frac{(2c+3)(n+6)(n-1)!c!}{2(n+c)!}\right) \frac{1}{2c+3}\frac{(n+c)!}{c!}\\\notag
&= \biggl(1+ \Theta\Big(\frac{1}{c}+ \frac{1}{n}\Big)\biggr) \frac{1}{2c+3}\frac{(n+c)!}{c!}\notag
\end{align*}
}
\end{theorem}
\begin{proof}

{\em Lower bound. } It follows from (\ref{eq:lower_bound_grado_nondiretto}). 

{\em Upper bound. }To prove the upper bound let $(A_0,A_1,\ldots, A_n)$  be the $A$-partition of $\bg{n+c}{n}{n}$. Using Lemma~\ref{lem:a-partition_undirected} for any $2\leq i\leq n-2$, there exist $S_i\subset A_i$ and  $S_{i+1} \subset A_{i+1}$ such that $S_i$ dominates $A_{i-1}\cup A_{i+1}$ and $S_{i+1}$ dominates $A_{i}\cup A_{i+2}$. Hence, the vertices in $S_i\cup S_{i+1}$ dominate all the vertices in $A_{i-1} \cup A_{i} \cup A_{i+1} \cup A_{i+2}$. We thus consider the sequence of sets $(A_1,\ldots, A_n)$ and we subdivide it in subsequences of exactly 4 sets except possibly the last one which may have less than $4$ sets.  For each such subsequence $(A_i, A_{i+1}, A_{i+2}, A_{i+3})$, using the previous argument we can dominate its vertices by taking set $S_{i+1} \cup S_{i+2}$ of $2\frac{(n+c-1)!}{(c+1)!}$ of vertices.  
If $n$ is not a multiple of $4$, we can distinguish two cases depending on the number of sets this subsequence contains: 
\begin{itemize}
    \item[(i)] the last subsequence contains only set $A_n$: it is sufficient to take $S_{n-1} \subset A_{n-1}$ of  $\frac{(n+c-1)!}{(c+1)!}$ vertices as indicated by Lemma~\ref{lem:a-partition_undirected}.
    
    \item[(ii)] the last subsequence contains two or three sets: we have $(A_{n-1}, A_{n})$ or $(A_{n-2},A_{n-1}, A_{n})$, respectively. 
    In both cases, by Lemma~\ref{lem:a-partition_undirected}, we can find $S_{n-1} \subset A_{n-1}$ that dominates $A_{n-2} \cup A_{n}$ and, by Lemma~\ref{lem:a-partition-directed}, we can find $S_n\subset A_{n}$ that dominates $A_{n-1}$. Thus, with $2\frac{(n+c-1)!}{(c+1)!}$ vertices in $S_{n-1}\cup S_n$ it is possible to dominate all the vertices in $A_{n-2}\cup A_{n-1} \cup A_{n}$.
\end{itemize}

So, all the vertices in $A_1\cup A_2 \cup \ldots, A_n$ can be dominated by a set of vertices $S'$ of cardinality: 
$$|S'|=\left\lceil\frac{n}{4}\right\rceil 2\frac{(n+c-1)!}{(c+1)! }\leq \frac{n+4}{2}\frac{(n+c-1)!}{(c+1)!}.$$ 

It remains to show how to dominate the vertices of $A_0$. 
Notice that the subgraph of $\bg{n+c}{n}{n}$ with $c\geq 1$ induced by the vertices in $A_0$, is isomorphic to the graph $\bg{n+c-1}{n}{n}$. Thus, to dominate $A_0$, we can iterate the previous argument. Let $T(n+c,n,n)$ be the number of vertices in a dominating set $\bg{n+c}{n}{n}$. We have:

$$T(n+c,n,n)\leq\left\{
\begin{array}{ll}
\left\lceil\frac{n}{3}\right\rceil(n-1)! & \mbox{if $c=0$}\\
T(n+c-1,n,n)+     \frac{n+4}{2}\frac{(n+c-1)!}{(c+1)!}      & \mbox{otherwise}\\
\end{array}
\right.
$$
where the case $c=0$ follows from  Theorem~\ref{theo:undirected_NNN}. Using substitution we have:
\begin{equation}\label{uno}
T(n+c,n,n)\leq \left\lceil\frac{n}{3}\right\rceil(n-1)! +\frac{n+4}{2} \sum_{i=1}^{c}\frac{(n+i-1)!}{(i+1)!}
\end{equation}

We use now the following equality that can be proved by induction. 
\begin{equation*}
\sum_{i=0}^h\frac{(n+i-1)!}{(i+1)!}=\frac{(h+n)!}{n(h+1)!}
\end{equation*}

Using the above equation we can solve the recurrence for $T(n+c,n,n)$ as follows.

\begin{align*}
T(n+c,n,n) & \leq \frac{n+4}{2}\frac{(n+c)!}{n(c+1)!} - \frac{n+4}{2}(n-1)! + \frac{n+3}{3}(n-1)! \nonumber\\
&=\left(1+ \frac{1}{2c+2}+ \frac{4}{n}+\frac{2}{n(c+1)} -  \frac{(2c+3)(n+6)(n-1)!c!}{2(n+c)!}\right) \frac{1}{2c+3}\frac{(n+c)!}{c!}\\\notag
&= \biggl(1+ \Theta\Big(\frac{1}{c}+ \frac{1}{n}\Big)\biggr) \frac{1}{2c+3}\frac{(n+c)!}{c!}\notag
\end{align*}
This concludes the proof. 
\end{proof}

We remark that, in the previous proof, we could have  simply taken the whole set $A_0$  to dominate the vertices in $A_0$. This leads to a dominating set of the following cardinality:

\begin{align*}
|S'|+|A_0| & \leq \frac{n+4}{2}\frac{(n+c-1)!}{(c+1)!}+ \frac{(n+c-1)!}{(c-1)!} \\
&=\biggl(1 + \frac{2c^2+c+2}{n+c} \biggr)  \frac{1}{2c+2} \frac{(n+c)!}{c!}
\end{align*}

For $c$ not constant, this upper bound is worse than the one claimed by Theorem~\ref{theo:undirected_NCNN}.

\paragraph*{Comparison with the directed case.} We conclude this section by observing that from Theorem~\ref{theo:directed_NCNN} we have $\gamma(\bg{n+c}{n}{n})\leq  \frac{(n+c-1)(n+c-1)!}{(c+1)!} = \frac{n+c-1}{(n+c)(c+1)} \frac{(n+c)!}{c!}$ which is worse that the result of Theorem~\ref{theo:undirected_NCNN} when $n$ goes to infinity. Furthermore, for small values of $n$ the results of this section show that, while for $n=2$, $\gamma(\bg{d}{2}{n})=\gamma(\dbg{d}{2}{n})=d-1$, for $n\geq 3$ the upper bound of $(d-1)^{n-1}$  derived from the directed case can be already improved to $(d-1)^{n-1} -(d-2)(d-1)^{n-4}$.

\section{Conclusions and open problems}

In this paper we introduce a new class of graphs, namely the $t$-constrained graphs, which are a natural generalization of de Bruijn and Kautz graphs and retain some of their structural properties. Under this new framework, de Bruijn graphs and Kautz graphs correspond to the cases $t=1$ and $t=2$, respectively. 
For these graphs we studied their domination number both in the directed and undirected case. 
Our work was motivated by the fact that while the domination number can be easily determined in the directed case for de Bruijn and Kautz graphs, the undirected case seems more difficult. 
Indeed, to the best of our knowledge, for these graphs, only the trivial bounds, deriving from (\ref{eq:lower_bound_grado_nondiretto}) and the directed case, were known. 
While we were able to improve the bounds for these two classes of graphs, the domination number of the $t$-constrained undirected de Bruijn graphs remains the main open problem of this paper. 
However, in the special case when the sequences labeling vertices are required to be  permutations (\textit{i.e.} $t=n$), we were able to exactly determine the domination number in both the directed and undirected case.

Finally, we believe the class of $t$-constrained de Bruijn graphs is worth being explored both in applied and theoretical context.

\acknowledgements
\label{sec:ack}
The authors want to thank the anonymous reviewers for their comments that helped to improve the manuscript.

\nocite{*}
\bibliographystyle{abbrvnat}
\bibliography{references}
\label{sec:biblio}

\end{document}